\documentclass[11pt,reqno]{amsart}
\usepackage{graphicx}
\usepackage{float}
\usepackage{amsfonts}
\usepackage[caption = false]{subfig}
\usepackage{amsmath}
\usepackage{graphics}
\usepackage{float}
\usepackage{enumerate}
\usepackage{mathtools}
\usepackage{amsthm}
\usepackage[english]{babel}
\usepackage{amsfonts,amssymb}% insert here the call for the packages your document requires
\usepackage{mathrsfs}
\usepackage[utf8x]{inputenc}\usepackage[english]{babel}
\usepackage{soul}
\usepackage[all]{xy,xypic}
\usepackage{cite}
\usepackage{relsize}
\numberwithin{equation}{section}
\newtheorem{theorem}{Theorem}[section]

\newtheorem{lemma}[theorem]{Lemma}

\theoremstyle{definition}

\theoremstyle{remark}
\newtheorem{remark}[theorem]{Remark}

\numberwithin{equation}{section}
\DeclareMathOperator{\RE}{Re}

\begin{document}
	%\fontsize{14pt}{16pt}\selectfont
	\title[On Hankel Determinants]{Certain Coefficient Problems of $\mathcal{S}_{e}^{*}$ and $\mathcal{C}_{e}$}
	%\title[sharp third hankel determinants]{Sharp third order Hankel determinants for starlike functions associated with the Sigmoid Function}

	\author[S. Sivaprasad Kumar]{S. Sivaprasad Kumar}
	\address{Department of Applied Mathematics, Delhi Technological University, Delhi--110042, India}
	\email{spkumar@dce.ac.in}
	
	\author{Neha Verma}
	\address{Department of Applied Mathematics, Delhi Technological University, Delhi--110042, India}
	\email{nehaverma1480@gmail.com}

	\subjclass[2010]{30C45, 30C50}
	
	\keywords{Coefficient estimate, Exponential, Starlike, Hankel determinants }
	\maketitle
\begin{abstract}
In this current study, we consider the classes $\mathcal{S}^{*}_{e}$ and $\mathcal{C}_e$ to obtain sharp bounds for the third Hankel determinant for functions within these classes. Additionally, we provide estimates for the sixth and seventh coefficients while establishing the fourth-order Hankel determinant as well.

\end{abstract}
\maketitle
	
\section{Introduction}
	\label{intro}
Consider the set of normalized analytic functions, denoted as $\mathcal{A}$, which are defined on the open unit disk $\mathbb{D}:=\{z\in \mathbb{C}: | z|<1\}$. These functions are represented by the expansion:
\begin{equation}
	f(z) = z+a_2z^2+a_3z^3+\cdots.\label{form}
	\end{equation}
Within this class, we define a subclass $\mathcal{S}$, which comprises univalent functions. Also, assume a class of analytic functions defined on the unit disk $\mathbb{D}$, which possess a positive real part. This class is represented as $\mathcal{P}$ whose elements are of the form $p(z)=1+\sum_{n=1}^{\infty}p_n z^n$. We use the notation $h_1 \prec h_2$ to indicate that function $h_1$ is subordinate to $h_2$, which implies the existence of a Schwarz function $w$ with the properties $w(0) = 0$ and $| w(z) | \leq | z |$, such that $h_1(z) = h_2(w(z))$.

The Bieberbach conjecture, as discussed in \cite[Page no. 17]{goodman vol1} has made a substantial contribution to the advancement of geometric function theory and the emergence of coefficient-related challenges. In the wake of this, numerous additional subclasses of $\mathcal{S}$, encompassing starlike functions denoted as $\mathcal{S}^{*}$ and convex functions denoted as $\mathcal{C}$, have been introduced. Notably, in 1992, Ma and Minda \cite{ma-minda} introduced the following two classes:
\begin{equation}
		\mathcal{S}^{*}(\varphi)=\bigg\{f\in \mathcal {A}:\dfrac{zf'(z)}{f(z)}\prec \varphi(z) \bigg\}\label{mindaclass}
\end{equation}
and
\begin{equation}
		\mathcal{C}(\varphi)=\bigg\{f\in \mathcal {A}:1+\dfrac{zf''(z)}{f'(z)}\prec \varphi(z) \bigg\},\label{mindaclassc}
\end{equation}
which unifies various subclasses of $\mathcal{S}^{*}$ and $\mathcal{C}$, respectively.
Here $\varphi$ is an analytic univalent function satisfying the conditions $\RE\varphi(z)>0$, $\varphi(\mathbb{D})$ symmetric about the real axis and starlike with respect to $\varphi(0)=1$ with $\varphi'(0)>0$.

The notion of Hankel determinants was introduced in \cite{pomi}. Remarkably, this concept continues to captivate the attention of numerous researchers to this very day. Encompassing a broad spectrum of applications and implications, the $qth$ Hankel determinants $H_q(n)$ of analytic functions belonging to the class $\mathcal{A}$, as represented in (\ref{form}), have been defined under the premise that $a_1$ takes the value $1$. For $n,q\in \mathbb{N}$, this definition unfolds as follows:
\begin{equation}
		H_{q}(n) =\begin{vmatrix}
			a_n&a_{n+1}& \ldots &a_{n+q-1}\\
			a_{n+1}&a_{n+2}&\ldots &a_{n+q}\\
			\vdots& \vdots &\ddots &\vdots\\
			a_{n+q-1}&a_{n+q}&\ldots &a_{n+2q-2}\label{5hqn}
		\end{vmatrix}
\end{equation}
The specific expression for the third-order Hankel determinant, denoted as $H_{3}(1)$, is obtained by substituting $q=3$ and $n=1$ into equation (\ref{5hqn}). This determinant can be precisely defined as:

\begin{equation}
H_{3}(1) =2 a_2a_3a_4-a_3^3-a_4^2-a_2^2a_5+a_3a_5.\label{1h3}
		%a_3(a_2a_4-a_3^2)-a_4(a_4-a_2a_3)+a_5(a_3-a_2^2)
\end{equation}
Over the time, several authors established sharp bound of second-order Hankel determinants, see \cite{alarif,krishna bezilevic}. However, the task of computing bounds for third-order Hankel determinants, proves to be considerably more intricate, can be observed from \cite{sharpstarlike,zap,zap2019}. In the context of the class $\mathcal{S}^{*}$, Kwon et al. \cite{sharpstarlike} established the inequality $| H_{3}(1)| \leq 8/9$, which has recently been best improved to the bound of $4/9$ by Kowalczyk et al. \cite{4/9}. Furthermore, Lecko et al. \cite{lecko 1/2 bound} successfully derived the bound $| H_{3}(1)| \leq 1/9$, a result that stands as sharp for functions in $\mathcal{S}^{*}(1/2)$. For a more comprehensive exploration of Hankel determinants, interested readers can turn to works such as \cite{sharp,4/9,lecko 1/2 bound,nehacardioid}.

%Zaprawa \cite{zap} estimated that

%\begin{equation*}
	%	|  H_{3}(1)| \leq \begin{cases}
			%1,& f\in \mathcal{S}^{*}\\
			%49/540,& f\in \mathcal{C}.
		%\end{cases}
%\end{equation*}
 
%Through the customization of the function $\varphi$ within the contexts of $\mathcal{S}^*(\varphi)$ and $\mathcal{C}(\varphi)$, we derive several recognized subcategories within both $\mathcal{S}^*$ and $\mathcal{C}$. 
%Below, we provide a couple of specific subclasses of $\mathcal{S}^*$ and $\mathcal{C}$, resulting from different choices of $\phi(z)$. %These subclasses are depicted below.
%If we choose $\varphi(z)=(1+Az)/(1+Bz)$, we obtain the suclasses $$
%Note that by taking $\varphi(z)=(1+Az)/(1+Bz)$, Janowski~\cite{1janowski} introduced, $\mathcal{S}^*[A,B]:=\mathcal{S}^*((1+Az)/(1+Bz))$ and $\mathcal{C}[A,B]:=\mathcal{C}((1+Az)/(1+Bz))$ for $-1\leq B<A\leq 1$.
%\begin{enumerate}
%\item On taking $\varphi(z)=(1+Az)/(1+Bz)$, Janowski~\cite{1janowski} introduced, $\mathcal{S}^*[A,B]:=\mathcal{S}^*((1+Az)/(1+Bz))$ and $\mathcal{C}[A,B]:=\mathcal{C}((1+Az)/(1+Bz))$ for $-1\leq B<A\leq 1$.
%\item On taking $\varphi(z)=\sqrt{1+z}$, Sok\'{o}\l \ and %Stankiewicz  \cite{stan} introduced the class, $\mathcal{S}^{*}_ L:=\mathcal{S}^{*}(\sqrt{1+z})$ and $\mathcal{C}_ L:=\mathcal{C}(\sqrt{1+z})$.
%\end{enumerate}
Below, we enlist specific subclasses of $\mathcal{S}^*$ and $\mathcal{C}$, resulting from diverse selections of $\varphi(z)$ in Table \ref{10 table}. In a similar manner, Mendiratta et al. \cite{mendi} introduced and analyzed the classes $\mathcal{S}_{e}^{*}$ and $\mathcal{C}_{e}$ by selecting $\varphi(z)=e^z$ in (\ref{mindaclass}) and (\ref{mindaclassc}), respectively. These classes are defined as follows:
\begin{equation*}
\mathcal{S}_{e}^{*}=\bigg\{f\in \mathcal {A}:\dfrac{zf'(z)}{f(z)}\prec e^{z}\bigg\}\quad \text{and}\quad \mathcal{C}_{e}=\bigg\{f\in \mathcal {A}:1+\dfrac{zf''(z)}{f'(z)}\prec e^{z}\bigg\}.
\end{equation*} 

\begin{table}[!htbp]\label{10 table}
\centering
\caption{List of subclasses of $\mathcal{S}^{*}$ and $\mathcal{C}$}
\begin{tabular}{|c|c|c |c|c|} 
 \hline
 $\mathcal{S}^{*}(\varphi)$ & $\mathcal{C}(\varphi)$&$\varphi(z)$ & Author(s) & Reference \\ [1ex] 
 \hline
  $\mathcal{S}^*[C,D]$ &$\mathcal{C}[C,D]$& $(1+Cz)/(1+Dz)$   & Janowski&\cite{1janowski}     \\
  \hline
  $\mathcal{S}^{*}_{SG}$&$\mathcal{C}_{SG}$ & $2/(1+e^{-z})$ & Goel and Kumar &\cite{goel} \\
  \hline
   $\mathcal{S}^{*}_{\varrho}$ &$\mathcal{C}_{\varrho}$& $1+ze^z$ & Kumar and Kamaljeet &\cite{kumar-ganganiaCardioid-2021}\\
   \hline
  % $\mathcal{S}^{*}_{e}$ & $e^z$ &  Mendiratta et al.&\cite{mendi} \\
  % \hline
    $\mathcal{S}^{*}_{q}$ & $\mathcal{C}_{q}$& $z+\sqrt{1+z^2}$ &  Raina and Sok\'{o}\l &\cite{raina} \\
    \hline
    $\mathcal{S}^{*}_ L$& $\mathcal{C}_ L$&$\sqrt{1+z}$ &  Sok\'{o}\l \ and Stankiewicz  &\cite{stan}   \\
   
 \hline
 
\end{tabular}
\label{table1}
\end{table}

Numerous studies have addressed radius problems \cite{mendi} and investigated implications of first and higher-order differential subordination \cite{adibastarlikenessexponential, nehadiffexpo} for the subclasses associated with the exponential function. Zaprawa \cite{zap2019} established bounds for the third Hankel determinants, yielding values of $0.385$ and $0.021$ for the classes $\mathcal{S}^{*}_{e}$ and $\mathcal{C}_{e}$, respectively, although the results were not sharp.

In our present investigation, we contribute by establishing sharp bounds for $H_3(1)$ for functions in the classes $\mathcal{S}^{*}_{e}$ and $\mathcal{C}_{e}$. Additionally, in the upcoming sections, we will provide estimations for the bounds of the sixth and seventh coefficients for the functions belonging to the classes, $\mathcal {S}^{*}_{e}$ and $\mathcal{C}_{e}$ and also evaluate the fourth Hankel determinant. %for functions belonging to .

   %%%%%%%%%%%%%%%%%%%%%%%%%%%%%%%%%%%%%%%%%%%%%%%%%%%%%%%%%%%%%%%%%%%%%%%%%%%%%%%%%%%%%%%%%%%%%%%%%%%%%%%%%%%%%%%%%%%%%%%%%%%%%%%%%%%%%%%%%%%%%%%%%%%%%%

\section{Hankel Determinants for $\mathcal{S}^{*}_{e}$}
				
\subsection{Preliminaries}
In this part of the section, we derive the expressions of $a_i$ $(i=2,3,\dots,7)$ in terms of Carath\'{eodory} coefficients. For this, let $f\in \mathcal{S}_{e}^{*},$ then there exists a Schwarz function $w(z)$ such that
\begin{equation}\label{5 formulaai}
	\dfrac{zf'(z)}{f(z)}=e^{w(z)}.
\end{equation}
Suppose that $p(z)=1+p_1z+p_2z^2+\cdots\in \mathcal{P}$ and consider $w(z)=(p(z)-1)/(p(z)+1)$. Further, by substituting the expansions of $w(z)$, $p(z)$ and $f(z)$ in equation (\ref{5 formulaai}) and then comparing the coefficients, we obtain the expressions of $a_i (i=2,3,...,7)$ in terms of $p_j (j=1,2,...,5)$, given as follows:
\begin{equation}\label{5 ea2}
a_2=\dfrac{1}{2}p_1,\quad a_3=\dfrac{1}{16}\bigg(4p_2+p_1^2\bigg), \quad a_4=\dfrac{1}{288}\bigg(-p_1^3 +12 p_1 p_2 + 48 p_3\bigg),
\end{equation}
\begin{equation}\label{5 ea5}
a_5=\dfrac{1}{1152}\bigg( p_1^4 -12 p_1^2 p_2 +24 p_1 p_3 + 144 p_4\bigg),
\end{equation}
\begin{equation}\label{5 ea6}
a_6=\frac{1}{57600}\bigg(-17 p_1^5 + 220 p_1^3 p_2 - 480 p_1 p_2^2 - 480 p_1^2 p_3 - 480 p_2 p_3 + 720 p_1 p_4 + 5760p_5\bigg),
\end{equation}
and
\begin{align}
a_7&=\frac{1}{8294400}\bigg(881 p_1^6 - 13260 p_1^4 p_2 + 48240 p_1^2 p_2^2 - 14400 p_2^3 + 29040 p_1^3 p_3\nonumber\\
&\quad\quad \quad \quad \quad\quad -106560 p_1 p_2 p_3- 57600 p_3^2 - 56160 p_1^2 p_4- 86400 p_2 p_4 \nonumber\\
&\quad\quad \quad \quad \quad\quad + 69120 p_1 p_5 \bigg).\label{5 ea7}
\end{align}
				
The formula for $p_i$ $(i=2,3,4)$, which is included in the Lemma \ref{pformula} below, plays a vital role in establishing the sharp bound for Hankel determinants and forms the foundation for our main results.
\begin{lemma}\cite{rj,lemma1}\label{pformula}
Let $p\in \mathcal {P}$ has the form $1+\sum_{n=1}^{\infty}p_n z^n.$ Then
\begin{equation}
2p_2=p_1^2+\gamma (4-p_1^2),\label{b2}
\end{equation}
\begin{equation}
4p_3=p_1^3+2p_1(4-p_1^2)\gamma -p_1(4-p_1^2) {\gamma}^2+2(4-p_1^2)(1-| \gamma| ^2)\eta, \label{b3}
\end{equation}
and \begin{align}
8p_4&=p_1^4+(4-p_1^2)\gamma (p_1^2({\gamma}^2-3\gamma+3)+4\gamma)-4(4-p_1^2)(1-| \gamma| ^2)(p_1(\gamma-1)\eta\nonumber\\
&\quad+\bar{\gamma}{\eta}^2-(1-| \eta| ^2)\rho), \label{b4}
\end{align}
for some $\gamma$, $\eta$ and $\rho$ such that $| \gamma| \leq 1$, $| \eta| \leq 1$ and $| \rho| \leq 1.$
\end{lemma}

\subsection{Sharp Third Hankel Determinant for $\mathcal{S}^{*}_{e}$}	

In this subsection, we present the sharp bound for $H_{3}(1) $ for functions belonging to the class $\mathcal{S}^{*}_{e}$.
\begin{theorem}\label{5 sharph31}
Let $f\in \mathcal {S}^{*}_{e}.$ Then
\begin{equation}
|  H_{3}(1)| \leq 1/9.\label{5 9.5}
\end{equation}
This result is sharp.
\end{theorem}

\begin{proof}
Since the class $\mathcal {P}$ is invariant under rotation, the value of $p_1$ belongs to the interval [0,2]. Let $p:=p_1$ and then substitute the values of $a_i(i=2,3,4,5)$ in equation (\ref{1h3}) from equations (\ref{5 ea2}) and (\ref{5 ea5}). We get
\begin{align*}
H_{3}(1)&=\dfrac{1}{331776}\bigg(-211 p^6 + 420 p^4 p_2 - 1872 p^2 p_2^2 - 5184 p_2^3 + 2544 p^3 p_3\\
&\quad \quad\quad \quad \quad \quad +10944 p p_2 p_3- 9216 p_3^2 - 7776 p^2 p_4 + 10368 p_2 p_4\bigg).
\end{align*}
After simplifying the calculations through (\ref{b2})-(\ref{b4}), we obtain
					
$$H_{3}(1)=\dfrac{1}{331776}\bigg(\beta_1(p,\gamma)+\beta_2(p,\gamma)\eta+\beta_3(p,\gamma){\eta}^2+\phi(p,\gamma,\eta)\rho\bigg),$$
for $\gamma,\eta,\rho\in \mathbb {D}.$ Here
\begin{align*}
\beta_1(p,\gamma):&=-13p^6-36{\gamma}^2p^2(4-p^2)^2-360{\gamma}^3p^2(4-p^2)^2+72{\gamma}^4p^2(4-p^2)^2\\
&\quad+78{\gamma}p^4(4-p^2)+120p^4{\gamma}^2(4-p^2)-324p^4{\gamma}^3(4-p^2)\\
&\quad-1296{\gamma}^2p^2(4-p^2),\\
\beta_2(p,\gamma):&=24(1-| \gamma| ^2)(4-p^2)(17p^3+54{\gamma}p^3+30p\gamma(4-p^2)-12p{\gamma}^2(4-p^2)),\\
\beta_3(p,\gamma):&=144(1-| \gamma| ^2)(4-p^2)(-16(4-p^2)-2| \gamma| ^2(4-p^2)+9p^2\bar{\gamma}),\\
\phi(p,\gamma,\eta):&=1296(1-| \gamma| ^2)(4-p^2)(1-| \eta| ^2)(2(4-p^2)\gamma-p^2).
\end{align*}
By choosing $x=| \gamma| $, $y=| \eta| $ and utilizing the fact that $| \rho| \leq 1,$ the above expression reduces to the following:
\begin{align*}
|  H_{3}(1)| \leq \dfrac{1}{331776}\bigg(| \beta_1(p,\gamma)| +| \beta_2(p,\gamma)|  y+| \beta_3(p,\gamma)|  y^2+| \phi(p,\gamma,\eta)| \bigg)\leq M(p,x,y),
\end{align*}
where
\begin{equation}
M(p,x,y)=\dfrac{1}{331776}\bigg(m_1(p,x)+m_2(p,x)y+m_3(p,x)y^2+m_4(p,x)(1-y^2)\bigg),\label{5 new}
\end{equation}
with
\begin{align*}
m_1(p,x):&=13p^6+36x^2p^2(4-p^2)^2+360x^3p^2(4-p^2)^2+72x^4p^2(4-p^2)^2\\
&\quad +78xp^4(4-p^2)+120p^4x^2(4-p^2)+324p^4x^3(4-p^2)+1296x^2p^2(4-p^2),\\
m_2(p,x):&=24(1-x^2)(4-p^2)(17p^3+54xp^3+30px(4-p^2)+12px^2(4-p^2)),\\
m_3(p,x):&=144(1-x^2)(4-p^2)(16(4-p^2)+2x^2(4-p^2)+9p^2x),\\
m_4(p,x):&=1296(1-x^2)(4-p^2)(2x(4-p^2)+p^2).
\end{align*}
				
In the closed cuboid $U:[0,2]\times [0,1]\times [0,1]$, we now maximise $M(p,x,y)$, by locating the maximum values in the interior of the six faces, on the twelve edges, and in the interior of $U$.
\begin{enumerate}
\item We start by taking into account every internal point of $U$. Assume that $(p,x,y)\in (0,2)\times (0,1)\times (0,1)$. We calculate $\partial{M}/\partial y$ %partially differentiate equation (\ref{3 new}) with respect to $y$
 to identify the points of maxima in the interior of $U$. We get
\begin{align*}
\dfrac{\partial M}{\partial y}&=\dfrac{(4 - p^2)(1 - x^2)}{13824}  \bigg(24 p x (5 + 2 x) + p^3 (17 + 24 x - 12 x^2)+96 (8 - 9 x + x^2) y \\
&\quad \quad\quad \quad\quad\quad\quad\quad\quad- 12 p^2 (25 - 27 x + 2 x^2) y\bigg).
\end{align*}
Now $\dfrac{\partial M}{\partial y}=0$ gives
\begin{equation*}
y=y_0:=\dfrac{p(17 p^2 + 120  x + 24 p^2 x + 48  x^2 - 12 p^2 x^2)}{12 (-64 + 25 p^2 + 72 x - 27 p^2 x - 8 x^2 + 2 p^2 x^2)}.
\end{equation*}
The existence of critical points requires that $y_0$ belong to $(0,1)$, which is only possible when
						
\begin{align}
300p^2+864x+24p^2x^2&>17p^3+120px+24p^3x+48px^2-12p^3x^2\nonumber\\
							&\quad+768+864x+24p^2x^2.\label{5 h1}
\end{align}
						
Now, we find the solution satisfying the inequality (\ref{5 h1}) %and (\ref{3 h2})
for the existence of critical points using the hit and trial method. If we assume $p$ tends to 0, then there does not exist any $x\in (0,1)$ satisfying the equation (\ref{5 h1}). But, when $p$ tends to 2, the equation (\ref{5 h1}) holds for all $x<37/54.$ We also observe that there does not exist any $p\in (0,2)$ when $x\in (37/54,1)$.
Similarly, if we assume $x$ tends to 0, then for all $p>1.68218$, the equation (\ref{5 h1}) holds. After calculations, we observe that there does not exist any $x\in (0,1)$ when $p\in (0,1.68218)$. Thus, the domain for the solution of the equation is $(1.68218,2)\times (0,37/54).$ Now, we examine that $\frac{\partial M}{\partial y}| _{y=y_0}\neq 0$ in $(1.68218,2)\times (0,37/54).$
So, we conclude that the function $M$ has no critical point in $(0,2)\times (0,1)\times (0,1).$

\item The interior of each of the cuboid $U$'s six faces is now being considered.\\
\underline{On $p=0$}, $M(p,x,y)$ turns into
\begin{equation}
s_1(x,y):=\dfrac{(1-x^2)(8y^2 + x^2y^2+9x(1-y^2))}{72},\quad x,y\in (0,1).\label{5 9.4}
\end{equation}
 Since
\begin{equation*}
\dfrac{\partial s_1}{\partial y}=\dfrac{(1 - x^2)(x-1)(x-8)y}{36}\neq 0,\quad x,y\in (0,1),
\end{equation*}
indicates that $s_1$ has no critical points in $(0,1)\times(0,1)$.	\\					
\noindent \underline{On $p=2$}, $M(p,x,y)$ reduces to
\begin{equation}
M(2,x,y):=\dfrac{13}{5184},\quad x,y\in (0,1).\label{5 9.3}
\end{equation}
\underline{On $x=0$}, $M(p,x,y)$ becomes
\begin{align}
s_2(p,y):&=\dfrac{13p^6 + (4-p^2)(408p^3y+2304y^2(4-p^2)+1296p^2(1-y^2)}{331776}
							\label{5 9.1}
\end{align}
with $p\in (0,2)$ and $y\in (0,1).$ To determine the points of maxima, we solve $\partial s_2/\partial p=0$ and $\partial s_2/\partial y=0$. After solving $\partial s_2/\partial y=0,$ we get
\begin{equation}
y=\dfrac{17p^3}{12(25p^2-64)}(=:y_p).\label{5 y}
\end{equation}
In order to have $y_p\in (0,1)$ for the given range of $y$, $p_0:=p>\approx 1.68218$ is required. Based on calculations, $\partial s_2/\partial p=0$ gives
\begin{equation}
1728 p - 864 p^3 + 13 p^5 + 816 p^2 y - 340 p^4 y - 7872 p y^2 +
							2400 p^3 y^2=0.\label{5 9}
\end{equation}
After substituting equation (\ref{5 y}) into equation (\ref{5 9}), we have
\begin{equation}
21233664 p - 27205632 p^3 + 11472192 p^5 - 1613016 p^7 + 2700 p^9=0.\label{5 40}
\end{equation}
A numerical calculation suggests that $p\approx 1.35596\in (0,2)$ is the solution of (\ref{5 40}). So, we conclude that $s_2$ does not have any critical point in $(0,2)\times(0,1)$.\\
						
\noindent \underline{On $x=1$}, $M(p,x,y)$ reforms into
\begin{equation}
s_3(p,y):=M(p,1,y)=\dfrac{12672 p^2 - 2952 p^4 - 41 p^6}{331776}, \quad p\in (0,2).\label{5 9.2}
\end{equation}
While computing $\partial s_3/\partial p=0$, $p_0:=p\approx 1.43461$ comes out to be the critical point. Undergoing simple calculations, $s_3$ achieves its maximum value $\approx 0.0398426$ at $p_0.$\\
						
\noindent \underline{On $y=0$}, $M(p,x,y)$ can be viewed as
\begin{align*}
s_4(p,x):&=\dfrac{1}{331776}\bigg(41472 x (1-x^2) + 576 p^2 (9 - 36 x + x^2 + 46 x^3 + 2 x^4)\\
&\quad \quad\quad\quad\quad\quad -24 p^4 (54 - 121 x - 8 x^2 + 174 x^3 + 24 x^4)\\
&\quad \quad\quad\quad\quad\quad +p^6 (13 - 78 x - 84 x^2 + 36 x^3 + 72 x^4)\bigg).
\end{align*}
After undergoing further calculations such as,
\begin{align*}
\dfrac{\partial s_4}{\partial x}&=\dfrac{1}{331776}\bigg(-82944 x^2 + 41472 (1-x^2) + 576 p^2 (-36 + 2 x + 138 x^2 + 8 x^3)\\
&\quad\quad \quad\quad \quad \quad-
24 p^4 (-121 - 16 x + 522 x^2 + 96 x^3) + p^6 (-78 - 168 x \\
&\quad\quad \quad\quad \quad \quad+ 108 x^2 + 288 x^3)\bigg)
\end{align*}
and \begin{align*}
\dfrac{\partial s_4}{\partial p}&=\dfrac{1}{331776}\bigg(6 p^5 (13 - 78 x - 84 x^2 + 36 x^3 + 72 x^4)-96 p^3 (54 - 121 x- 8 x^2 \\
&\quad\quad \quad\quad\quad  + 174 x^3+ 24 x^4)+1152 p (9 - 36 x + x^2 + 46 x^3 + 2 x^4)  \bigg),
\end{align*}
we observe that no solution in $(0,2)\times (0,1)$ exists of the system of equations $\partial s_4/\partial x=0$ and $\partial s_4/\partial p=0$.\\
\noindent \underline{On $y=1$}, $M(p,x,y)$ reduces to
\begin{align*}
s_5(p,x):&=\dfrac{1}{331776}\bigg(2304 p x (5 + 2 x - 5 x^2 - 2 x^3) - 4608 (-8 + 7 x^2 + x^4)\\
&\quad\quad \quad\quad\quad+576 p^2 (-32 + 9 x + 38 x^2 + x^3 + 6 x^4)-24 p^5 (17 + 24 x \\
&\quad\quad \quad\quad\quad - 29 x^2- 24 x^3 + 12 x^4)+96 p^3 (17 - 6 x - 41 x^2+ 6 x^3\\
&\quad\quad \quad\quad\quad  + 24 x^4)-24 p^4 (-96 + 41 x + 130 x^2+ 12 x^3 + 36 x^4)\\
&\quad\quad \quad\quad\quad +p^6 (13 - 78 x - 84 x^2 + 36 x^3+ 72 x^4)\bigg).
\end{align*}
						
\noindent The system of equations $\partial s_5/\partial x=0$ and $\partial s_5/\partial p=0$ also do not have any solution in $(0,2)\times (0,1).$\\

\item We next examine the maxima attained by $M(p,x,y)$ on the edges of the cuboid $U$. From equation (\ref{5 9.1}), we have $M(p,0,0)=r_1(p):=(5184p^2-1296p^4+13p^6)/331776.$ It is easy to observe that $r_1'(p)=0$ whenever $p=\delta_0:=0$ and $p=\delta_1:=1.4367\in [0,2]$ as its points of minima and maxima respectively. %The maximum value of $r_1(p)$ is $\approx 0.0159535.$
Hence,
\begin{equation*}
M(p,0,0)\leq 0.0159535, \quad p\in [0,2].
\end{equation*}
Now considering the equation (\ref{5 9.1}) at $y=1,$ we get $M(p,0,1)=r_2(p):=(36864 - 18432 p^2 + 1632 p^3 + 2304 p^4 - 408 p^5 + 13 p^6)/331776.$ It is easy to observe that $r_2'(p)<0$ in $[0,2]$ and hence $p=0$ serves as the point of maxima. So,
\begin{equation*}
M(p,0,1)\leq \dfrac{1}{9}, \quad p\in [0,2].
\end{equation*}
Through computations, equation (\ref{5 9.1}) shows that $M(0,0,y)$ attains its maxima at $y=1.$ This implies that
\begin{equation*}
M(0,0,y)\leq \dfrac{1}{9}, \quad y\in [0,1].
\end{equation*}
Since, the equation (\ref{5 9.2}) does not involve $x$, we have $M(p,1,1)=M(p,1,0)=r_3(p):=(12672p^2 - 2952 p^4-41 p^6)/331776.$ Now, $r_3'(p)=4224 p - 1968 p^3 - 41 p^5=0$ when $p=\delta_2:=0$ and $p=\delta_3:=1.43461$ in the interval $[0,2]$ with $\delta_2$ and $\delta_3$ as points of minima and maxima respectively. Hence
\begin{equation*}
M(p,1,1)=M(p,1,0)\leq 0.0398426,\quad p\in [0,2].
\end{equation*}
After considering $p=0$ in (\ref{5 9.2}), we get, $M(0,1,y)=0.$ The equation (\ref{5 9.3}) has no variables. So, on the edges, %$p=2, x=1$; $p=2, x=0$; $p=2, y=0$ and $p=2, y=1,$ respectively,
the maximum value of $M(p,x,y)$ is
\begin{equation*}
M(2,1,y)=M(2,0,y)=M(2,x,0)=M(2,x,1)=\dfrac{13}{5184},\quad x,y\in [0,1].
\end{equation*}

Using equation (\ref{5 9.4}), we obtain $M(0,x,1)=r_4(x):=(8 - 7 x^2 - x^4)/72.$ Upon calculations, we see that $r_4(x)$ is a decreasing function in $[0,1]$ and attains its maxima at $x=0.$ Hence
\begin{equation*}
M(0,x,1)\leq \dfrac{1}{9},\quad x\in [0,1].
\end{equation*}
Again utilizing the equation (\ref{5 9.4}), we get $M(0,x,0)=r_5(x):=x(1-x^2)/8.$ On further calculations, we get $r_5'(x)=0$ for $x=\delta_4:=1/\sqrt{3}.$ Also, $r_5(x)$ is an increases in $[0,\delta_4)$ and decreases in $(\delta_4,1].$ So, it reaches its maximum value at $\delta_4.$ Thus
\begin{equation*}
M(0,x,0)\leq 0.0481125,\quad x\in [0,1].
\end{equation*}
\end{enumerate}
Given all the cases, the inequality (\ref{5 9.5}) holds.\\
Let the function
$f_1(z)\in \mathcal{S}^{*}_{e}$, be defined as
\begin{equation*}
f_1(z)=z\exp\bigg(\int_{0}^{z}\dfrac{e^{t^3}-1}{t}dt\bigg)=z+\dfrac{z^4}{3}+\dfrac{5z^7}{36}+\cdots,\label{5 extremal}
\end{equation*}
with $f_1(0)=0$ and $f_1'(0)=1$, acts as an extremal function for the bound of $|  H_{3}(1)| $ for $a_2=a_3=a_5=0$ and $a_4=1/3$.				
\end{proof}

\subsection{Fourth Hankel Determinant for $\mathcal{S}^{*}_{e}$}
In this subsection, we derive the bounds of sixth and seventh coefficients and consequently $H_4(1)$ for functions belonging to the class $\mathcal{S}^{*}_{e}$. We need the following lemma for deriving our results.

\begin{lemma}\cite{bellv,shelly}\label{2 pomi lemma}
Let $p=1+\sum_{n=1}^{\infty}p_nz^n\in \mathcal{P}.$ Then
\begin{equation*}
|  p_n| \leq 2, \quad n\geq 1,\label{2 caratheodory1}
\end{equation*}
\begin{equation*}
|  p_{n+k}-\nu p_n p_k| \leq \begin{cases}
2, & 0\leq \nu\leq 1;\\
2| 2\nu-1| ,& otherwise,
\end{cases}\label{2 caratheodory2}
\end{equation*}
and \begin{equation*}
|  p_1^3-\nu p_3| \leq
\begin{cases}2| \nu-4| ,& \nu\leq 4/3;\\ \\
2\nu\sqrt{\dfrac{\nu}{\nu-1}},& 4/3<\nu.
\end{cases}\label{2 caratheodory3}
\end{equation*}
\end{lemma}

We derive the expression of the fourth Hankel determinant when $q=4$ and $n=1$ are put into equation (\ref{5hqn}) as follows :
\begin{equation}\label{5 h41}
	H_{4}(1)=a_7H_{3}(1)-a_6T_1+a_5T_2-a_4T_3,
\end{equation}
where
\begin{equation}\label{5t1}
T_1:=a_6(a_3-a_2^2)+a_3(a_2a_5-a_3a_4)-a_4(a_5-a_2a_4),
\end{equation}
\begin{equation}\label{5t2}
T_2:=a_3(a_3a_5-a_4^2)-a_5(a_5-a_2a_4)+a_6(a_4-a_2a_3),
\end{equation}
and
\begin{equation}\label{5t3}
T_3:=a_4(a_3a_5-a_4^2)-a_5(a_2a_5-a_3a_4)+a_6(a_4-a_2a_3).\end{equation}
				
\noindent Now, using Lemma \ref{2 pomi lemma}, we first determine the bounds of $T_1$, $T _2$, and $T_3$.\\ By substituting the values of $a_i$'s $(i=2,3,...,6)$ in (\ref{5t1}) using (\ref{5 ea2})-(\ref{5 ea6}), we obtain
\begin{align*}
5529600T_1&=581 p_1^7+ 5040 p_1^4 p_3 + 25920 p_1^2 p_2 p_3- 7068 p_1^5 p_2  + 11040 p_1^3 p_4 \\
& \quad -115200 p_3 p_4+7920 p_1^3 p_2^2- 69120 p_2^2 p_3  +74880 p_1 p_2 p_4-25920 p_1 p_2^3  \\
&\quad+ 57600 p_1 p_3^2+ 138240 p_2 p_5-103680 p_1^2 p_5
\end{align*}
or
\begin{align*}
5529600| T_1| &\leq |  p_1^4(581 p_1^3+ 5040 p_3)|  + | p_1^2 p_2(25920  p_3- 7068 p_1^3) | + | 57600 p_1 p_3^2| \nonumber \\
& \quad+|  p_2^2(7920 p_1^3 - 69120 p_3)|   + |  p_1 p_2( 74880 p_4-25920 p_2^2)| \nonumber\\
&\quad +| p_4 ( 11040 p_1^3 -115200 p_3 )| + |  p_5(138240 p_2- 103680 p_1^2)|.%\label{5 et1}
\end{align*}
Using Lemma \ref{2 pomi lemma} and the triangle inequality, we arrive at
%\begin{align}
%| p_1^4(581 p_1^3+ 5040  p_3)| \leq 235648, \quad    | p_1^2 p_2(25920  p_3- 7068 p_1^3) | \leq 4976640\sqrt{\frac{15}{1571}},\label{5 et11}
%\end{align}
%\begin{equation}  | p_5(138240 p_2 - 103680 p_1^2)|  \leq 552960\quad  |  p_2^2(7920 p_1^3 - 69120 p_3)| \leq 442368\sqrt{\frac{30}{17}},\label{5 et14}\end{equation}
%\begin{align}
%| p_1 p_2 (74880 p_4-25920 p_2^2)| \leq 599040, \quad | 57600 p_1 p_3^2| \leq 460800, \label{5 et13}
%\end{align}

%and
%\begin{align}
%| p_4(11040 p_1^3 -115200 p_3) |  \leq 1843200\sqrt{\frac{15}{217}}. \label{5 et12}
%\end{align}
%By considering equation (\ref{5 et1}) in view of (\ref{5 et11})-(\ref{5 et12}),
\begin{align*}
| T_1| &\leq \frac{1848448 + 4976640 \sqrt{\frac{15}{1571}} + 1843200 \sqrt{\frac{15}{217}} + 442368 \sqrt{\frac{30}{17}} }{5529600} \nonumber\\
&\approx 0.616137.%\label{5 et1value}
\end{align*}
				
\noindent Now, we calculate the bound of $T_2$ in the similar way by substituting the values of $a_i$'s $(i=2,3,...,6)$ in (\ref{5t2}) from equations (\ref{5 ea2})-(\ref{5 ea6}), as follows:
				% BOUND OF T2
\begin{align}
22118400T_2&= 235 p_1^8+ 8712 p_1^5 p_3 + 37440 p_1^3 p_2 p_3- 1156 p_1^6 p_2 -63360 p_1 p_2^2 p_3\nonumber\\
&\quad -14640 p_1^4 p_2^2+ 161280 p_1 p_3 p_4 - 8400 p_1^4 p_4 + 368640 p_3 p_5 \nonumber \\
&\quad- 76800 p_1^3 p_5- 8640 p_1^2 p_2^3+172800 p_2^2 p_4 - 345600 p_4^2 - 40320 p_1^2 p_3^2\nonumber\\
&\quad- 184320 p_2 p_3^2+178560 p_1^2 p_2 p_4-184320 p_1 p_2 p_5\nonumber
\end{align}
or
\begin{align*}
22118400| T_2| &\leq | p_1^5(235 p_1^3+ 8712 p_3)|  + | p_1^3 p_2(37440  p_3- 1156 p_1^3)| +| 8640 p_1^2 p_2^3| \nonumber\\
&\quad+|  p_1 p_2^2(63360 p_3+14640 p_1^3)|  + | p_1p_4(161280p_3 - 8400 p_1^3) |  \nonumber \\
&\quad + | p_5(368640 p_3 - 76800 p_1^3)| +| p_4(172800 p_2^2 - 345600 p_4)| \nonumber\\
&\quad+| p_3^2(184320 p_2 + 40320 p_1^2) | +| p_1 p_2(178560 p_1 p_4-184320  p_5) | . %\label{5 et2}
\end{align*}
				
Lemma \ref{2 pomi lemma} and the triangle inequality lead us to
\begin{align*}
| T_2| &\leq \frac{7821568 + 14376960 \sqrt{\frac{65}{9071}} + 2949120 \sqrt{\frac{6}{19}} + 737280 \sqrt{\frac{42}{13}}}{22118400} \nonumber\\
&\approx 0.543487.%\label{5 et2value}
\end{align*}
%\begin{equation}
%\begin{aligned}
% \left.
%\begin{array}{cc}
%    & | p_1^5(235 p_1^3+ 8712  p_3)| \leq 617728;\\
% &|  p_1 p_2^2(63360 p_3+14640 p_1^3)| \leq 1950720;\\
%&| p_3^2(184320 p_2+40320 p_1^2 ) | \leq 2119680;\\
%& | p_4(172800 p_2^2  - 345600 p_4)| \leq 1382400; \label{5 et21}
%\end{array}
%\right\}
%\end{aligned}
%\end{equation}

%and 

%\begin{equation}
%\begin{aligned}
% \left.
%\begin{array}{cc}
%    &| 8640 p_1^2 p_2^3| \leq 276480;\\
%    &| p_1 p_2(178560 p_1 p_4-184320  p_5)| \leq  1474560;\\
% &| | p_1^3 p_2(37440  p_3- 1156 p_1^3)| \leq 14376960 \sqrt{\frac{65}{9071}};\\
%&| p_5(368640 p_3- 76800 p_1^3)| \leq 2949120 \sqrt{\frac{6}{19}};\\
%& | p_1p_4(161280 p_3 - 8400 p_1^3)| \leq 737280 \sqrt{\frac{42}{13}}.\label{5 et25}
%\end{array}
%\right\}
%\end{aligned}
%\end{equation}

%\begin{align}
%\quad | p_1^3 p_2(37440  p_3- 1156 p_1^3)| \leq 14376960 %\sqrt{\frac{65}{9071}},\label{5 et21}
%\end{align}
%\begin{align}& ;\\
%|  p_1 p_2^2(63360 p_3+14640 p_1^3)| \leq 1950720,\quad | 8640 p_1^2 p_2^3| \leq 276480, \label{5 et22}
%\end{align}
%\begin{align}
%| p_5(368640 p_3- 76800 p_1^3)| \leq 2949120 \sqrt{\frac{6}{19}}, \quad \label{5 et23}
%\end{align}
%\begin{align}
%| p_1p_4(161280 p_3 - 8400 p_1^3)| \leq 737280 \sqrt{\frac{42}{13}},\quad
%,    \label{5 et24}
%\end{align}
%and
%\begin{align}
%,\quad| p_1 p_2(178560 p_1 p_4-184320  p_5)| \leq  1474560.\label{5 et25}
%\end{align}
%By substituting the values from equations (\ref{5 et21}) and (\ref{5 et25}) in equation (\ref{5 et2}),		
Next, we determine the bound of $T_3$, by replacing the values of $a_i$'s $(i=2,3,...,6)$ from equations (\ref{5 ea2})-(\ref{5 ea6}) in (\ref{5t3}), as follows:
\begin{align*}
597196800T_3&=6120 p_1^8+ 143424 p_1^5 p_3 - 425 p_1^9-9000 p_1^6 p_3+ 9000 p_1^7 p_2\\
&\quad + 172800 p_1^4 p_2 p_3+ 302400 p_1^3 p_3^2- 2764800 p_3^3+1036800 p_1^3 p_2 p_4\\
&\quad+6220800 p_2 p_3 p_4-17280 p_1^4 p_2^2+ 9953280 p_3 p_5- 2073600 p_1^3 p_5 \\
&\quad+ 967680 p_1^3 p_2 p_3-64512 p_1^6 p_2-1036800 p_1 p_2 p_3^2- 32400 p_1^5 p_2^2\\
&\quad-777600 p_1^2 p_2^2 p_3 + 1244160 p_1 p_3 p_4 -259200 p_1^4 p_4-97200 p_1^5 p_4 \\
&\quad+1555200 p_1 p_2^2 p_4- 4665600 p_1 p_4^2- 414720 p_1 p_2^2 p_3-172800 p_1^3 p_2^3\\
&\quad-829440 p_2 p_3^2- 829440 p_1^2 p_3^2+414720 p_1^2 p_2^3- 622080 p_1^2 p_2 p_4\\
&\quad- 4976640 p_1 p_2 p_5
\end{align*}
or
\begin{align*}
597196800| T_3| &\leq | p_1^5(6120 p_1^3+143424 p_3)| +| p_1^6( 425 p_1^3+9000 p_3)| +| 17280 p_1^4 p_2^2| \nonumber\\
&\quad + | p_1^4 p_2(9000 p_1^3+ 172800p_3)| + |  p_3^2(302400 p_1^3- 2764800 p_3)| \nonumber\\
&\quad+| p_2 p_4(1036800 p_1^3 +6220800 p_3)| + | p_5(9953280 p_3-2073600 p_1^3)| \nonumber\\
&\quad + | p_1^3 p_2 (967680 p_3-64512 p_1^3)| +| 1036800 p_1 p_2 p_3^2| +| 97200 p_1^5 p_4 | \nonumber\\
&\quad +| p_1^2 p_2^2(32400 p_1^3+777600  p_3)| + | p_1p_4(1244160 p_3-259200 p_1^3)| \nonumber \\
&\quad + | p_1p_4(1555200 p_2^2 - 4665600 p_4)| +| p_1^2 p_2(414720p_2^2- 622080  p_4)|  \nonumber\\
&\quad+| p_3^2(829440 p_2 + 829440 p_1^2)| +| 172800 p_1^3 p_2^3| \nonumber\\
&\quad+| p_1 p_2( 414720 p_2 p_3+ 4976640  p_5)|.%\label{5 et3}
\end{align*}
By applying Lemma \ref{2 pomi lemma} and the triangle inequality, 
\begin{align*}
| T_3| &\leq \frac{286061056 + 58982400 \sqrt{\frac{3}{19}} + 99532800 \sqrt{\frac{6}{19}} + 2211840 \sqrt{210}}{597196800}\nonumber\\
&\approx 0.665582.%\label{5 et3value}
\end{align*}

\begin{remark}\label{5st}
On the basis of the above calculations, the bounds of $T_1$, $T_2$ and $T_3$ are $0.616137$, $0.543487$ and $0.665582$ respectively.
\end{remark}
				
To progress further, our next objective is to determine the bounds of the initial coefficients $a_i$ where $i=2,3,4,5$. These bounds, as derived in \cite{zap2019}, are summarized in the following remark.
\begin{remark}\label{5 zapa5s}
For $f\in \mathcal{S}^{*}_{e},$ $| a_2| \leq 1,$ $| a_3| \leq 3/4$, $| a_4| \leq 17/36$ and $| a_5| \leq 25/72$. Here the first three bounds are sharp.
\end{remark}

Finding coefficient bounds for $n>5$ becomes notably more challenging. In order to overcome this difficulty, we employ Lemma \ref{2 pomi lemma} to deduce the bounds for the sixth and seventh coefficients within the class of functions $\mathcal{S}^{*}_{e}$, as demonstrated in the subsequent lemma.

\begin{lemma}\label{5 a6a7lemma}
Let $f\in \mathcal{S}_{e}^{*}.$ Then $| a_6| \leq 587/1800\approx 0.326111$ and $| a_7| \leq 1397/4320\approx 0.32338$.
\end{lemma}
\begin{proof}
By suitably rearranging the terms given in equation (\ref{5 ea6}), we have
\begin{equation*}
57600a_6= 220 p_1^3 p_2 - 480 p_1^2 p_3 - 480 p_1 p_2^2 + 720 p_1 p_4-17 p_1^5- 480 p_2 p_3+ 5760p_5.
\end{equation*}
Using triangle inequality, it can be viewed as
\begin{align}
57600| a_6| &\leq | p_1^2(220 p_1 p_2-480 p_3) | + | p_1( 720 p_4 - 480p_2^2)| +| -17 p_1^5| \nonumber\\
&\quad +| 5760p_5- 480 p_2 p_3| .\label{5 a6}
\end{align}
Using Lemma \ref{2 pomi lemma}, we arrive at the following inequality:
%\begin{equation}
%| p_1^2(220 p_1 p_2-480 p_3) | \leq 3840,\quad | p_1( 720 p_4 - 480p_2^2)| \leq 2880, \quad | 17 p_1^5| \leq 544,\label{5 a61} \end{equation}
%and
%\begin{equation}
%| 5760p_5- 480 p_2 p_3| \leq 11520. \label{5 a62}
%\end{equation}
%By using equation (\ref{5 a61}) and (\ref{5 a62}) in equation (\ref{5 a6}), we obtain
\begin{equation*}
| a_6| \leq \frac{587}{1800}\approx 0.326111.				
\end{equation*}
Similarly, considering equation (\ref{5 ea7}), we have
\begin{align*}
8294400a_7&=881 p_1^6 - 13260 p_1^4 p_2 + 48240 p_1^2 p_2^2 - 14400 p_2^3 + 29040 p_1^3 p_3- 56160 p_1^2 p_4\nonumber\\
&\quad +69120 p_1 p_5-106560 p_1 p_2 p_3 - 57600 p_3^2  - 86400 p_2 p_4.
\end{align*}
Through the triangle inequality, it can also be seen as
\begin{align*}
8294400| a_7| &\leq| p_1^4(881 p_1^2 - 13260  p_2)|   + | p_2^2 (48240 p_1^2 - 14400 p_2)| \nonumber\\
&\quad+|  p_1 (69120 p_5-106560p_2 p_3)| + |  p_1^2(29040 p_1 p_3- 56160 p_4)| \nonumber\\
&\quad+|  57600 p_3^2| +|    86400 p_2 p_4 | .%\label{5a7}
\end{align*}
Lemma \ref{2 pomi lemma} implies that $|a_7|\leq 1397/4320\approx 0.32338.$
%\begin{equation*}
%| a_7| \leq %\frac{2682240}{8294400}
%\frac{1397}{4320}\approx 0.32338.
%\end{equation*}
\end{proof}
%\begin{align}
%| p_1^4(881 p_1^2 - 13260 p_2)| \leq 424320,\quad   | p_2^2(48240 p_1^2  - 14400 p_2)| \leq 656640, \label{5 ea71}
%\end{align}
%\begin{align}\label{5 ea72}
%| p_1(69120  p_5-106560 p_2 p_3)| \leq 576000, \quad | p_1^2(29040 p_1 p_3- 56160  p_4)| \leq 449280,
%\end{align}
%and \begin{align}\label{5 ea73}
%|  57600 p_3^2 | +|   86400 p_2 p_4 | \leq 576000.
%\end{align}
%By substituting the values from equations (\ref{5 ea71})-(\ref{5 ea73}) in (\ref{5a7}), we have

\begin{theorem}\label{5 h41bound}
Let $f\in \mathcal{S}^{*}_{e}.$ Then
\begin{equation*}
| H_{4}(1)| \leq 0.29059.\end{equation*}
\end{theorem}
%\begin{proof}
The proof of the above theorem follows by substituting the values obtained from Theorem \ref{5 sharph31}, %equations (\ref{5 et1value})-(\ref{5 et3value}),
Remark \ref{5st}, Remark \ref{5 zapa5s} and Lemma \ref{5 a6a7lemma} in the equation (\ref{5 h41}), therefore, it is skipped here.
%\end{proof}

%%%%%%%%%%%%%%%%%%%%%%%%%%%%%%%%%%%%%%%%%%% For the class of convex functions %%%%%%%%%%%%%%%%%%%%%%%%%%%%%%%%%
\section{Hankel Determinants for $\mathcal{C}_{e}$}
				
\subsection{Preliminaries}
In this segment, we express the expressions of initial coefficients $a_i$ $(i=2,3,\ldots,7)$ involving Carath\'{eodory} coefficients. When $f\in \mathcal{C}_{e}$, we replace the L.H.S of equation (\ref{5 formulaai}) by $1+zf''(z)/f'(z)$ and arrive at the following equation
%We get the following equation where $f(z)$ is described in equation (\ref{form}) and $w(z)$.
				
\begin{equation*}\label{5 formulaaic}
	1+\dfrac{zf''(z)}{f'(z)}=e^{w(z)}.
		\end{equation*}
Proceeding on the similar lines as done for the class $\mathcal{S}^{*}_{e}$, we obtain $a_i (i=2,3,...,7)$ in terms of $p_j (j=1,2,...,5)$, then compare the corresponding coefficients as follows:
%The initial coefficients are given as
\begin{equation}\label{5 ea2c}
a_2=\dfrac{1}{4}p_1,\quad a_3=\dfrac{1}{48}\bigg(p_1^2+4p_2\bigg), \quad a_4=\dfrac{1}{1152}\bigg(-p_1^3 +12 p_1 p_2 + 48 p_3\bigg),
\end{equation}
\begin{equation}\label{5 ea5c}
a_5=\dfrac{1}{5760}\bigg( p_1^4 -12 p_1^2 p_2 +24 p_1 p_3 + 144 p_4\bigg),
\end{equation}
\begin{align}
a_6&=\frac{1}{345600}\bigg(-17 p_1^5 + 220 p_1^3 p_2 - 480 p_1 p_2^2 - 480 p_1^2 p_3 - 480 p_2 p_3+720 p_1 p_4\nonumber\\
&\quad\quad\quad\quad\quad+ 5760 p_5\bigg),\label{5 ea6c}
\end{align}
and
\begin{align}
a_7&=\frac{1}{58060800}\bigg(881 p_1^6 - 13260 p_1^4 p_2 + 48240 p_1^2 p_2^2 - 14400 p_2^3 +29040 p_1^3 p_3- 106560 p_1 p_2 p_3 \nonumber\\\label{5 ea7c}
&\quad\quad\quad\quad\quad\quad  - 57600 p_3^2 - 56160 p_1^2 p_4- 86400 p_2 p_4+ 69120 p_1 p_5\bigg).
\end{align}
\subsection{Sharp Third Hankel Determinant for $\mathcal{C}_{e}$}
In this subsection, we establish the sharp bound of $H_{3}(1)$ for functions that belong to the class $\mathcal{C}_{e}$.
				
\begin{theorem}\label{5 ceh31}
Let $f\in \mathcal {C}_{e}.$ Then
\begin{equation}
|  H_{3}(1)| \leq \dfrac{1}{144}.\label{5 sharph31c}
\end{equation}
This bound is sharp.
\end{theorem}
\begin{proof}
We follow the same steps which were used to prove Theorem \ref{5 sharph31}. The values of $a _i's(i=2,3,4,5)$ from equations (\ref{5 ea2c}) and (\ref{5 ea5c}) are substituted into equation (\ref{1h3}). Thus

\begin{align*}
H_{3}(1)&=\dfrac{1}{6635520}\bigg(-173 p^6 + 552 p^4 p_2 - 1872 p^2 p_2^2 - 3840 p_2^3 + 2208 p^3 p_3 \\
&\quad\quad \quad \quad \quad \quad+ 8064 p p_2 p_3-11520 p_3^2 - 6912 p^2 p_4 + 13824 p_2 p_4\bigg).
\end{align*}
Using (\ref{b2})-(\ref{b4}) for simplification, we arrive at
$$H_{3}(1)=\dfrac{1}{6635520}\bigg(\alpha_1(p,\gamma)+\alpha_2(p,\gamma)\eta+\alpha_3(p,\gamma){\eta}^2+\psi(p,\gamma,\eta)\rho\bigg),$$
where $\gamma,\eta,\rho\in \mathbb {D},$
\begin{align*}	\alpha_1(p,\gamma):&=-5p^6-180{\gamma}^2p^2(4-p^2)^2+1536{\gamma}^3(4-p^2)^2-240{\gamma}^3p^2(4-p^2)^2\\
&\quad+144{\gamma}^4p^2(4-p^2)^2+12{\gamma}p^4(4-p^2)-120p^4{\gamma}^2(4-p^2),\\
\alpha_2(p,\gamma):&=(1-| \gamma| ^2)(4-p^2)(240p^3-288p{\gamma}(4-p^2)-576p\gamma^2(4-p^2)),\\
\alpha_3(p,\gamma):&=(1-| \gamma| ^2)(4-p^2)(-2880(4-p^2)-576| \gamma| ^2(4-p^2)),\\
\psi(p,\gamma,\eta):&=3456\gamma(1-| \gamma| ^2)(4-p^2)^2(1-| \eta| ^2).
\end{align*}
Since $| \rho| \leq 1$, also for the simplicity of the calculations, assume $x=| \gamma| $ and $y=| \eta| $,
%Additionally, since $| \rho| \leq 1,$ and taking $x=| \gamma| $, $y=| \eta| $,
\begin{align*}
|  H_{3}(1)| \leq \dfrac{1}{6635520}\bigg(| \alpha_1(p,\gamma)| +| \alpha_2(p,\gamma)| y+| \alpha_3(p,\gamma)| y^2+| \psi(p,\gamma,\eta)| \bigg)\leq N(p,x,y),
\end{align*}
where
\begin{equation}
N(p,x,y)=\dfrac{1}{6635520}\bigg(n_1(p,x)+n_2(p,x)y+n_3(p,x)y^2+n_4(p,x)(1-y^2)\bigg),\label{3s new}
\end{equation}
with
\begin{align*}
n_1(p,x):&=5p^6+180x^2p^2(4-p^2)^2+1536x^3(4-p^2)^2+240x^3p^2(4-p^2)^2\\
&\quad +144x^4p^2(4-p^2)^2+12xp^4(4-p^2)+120p^4x^2(4-p^2),\\
n_2(p,x):&=(1-x^2)(4-p^2)(240p^3+288px(4-p^2)+576px^2(4-p^2)),\\
n_3(p,x):&=(1-x^2)(4-p^2)(2880(4-p^2)+576x^2(4-p^2)),\\
n_4(p,x):&=3456x(1-x^2)(4-p^2)^2.
\end{align*}
					
We must maximise $N(p,x,y)$ in the closed cuboid $V:[0,2]\times [0,1]\times [0,1]$. By identifying the maximum values on the twelve edges, the interior of $V$, and the interiors of the six faces, we can prove this.
\begin{enumerate}
\item We start by taking into account, every interior point of $V$. Assume that $(p,x,y)\in (0,2)\times (0,1)\times (0,1).$ We partially differentiate equation (\ref{3s new}) with respect to $y$ to locate the points of maxima in the interior of $V$. We obtain
\begin{align*}
\dfrac{\partial N}{\partial y}&=\dfrac{(1 - x^2)(4 - p^2) }{138240} \bigg(24 p x (1 + 2 x)-p^3 (-5 + 6 x + 12 x^2)+ 96 (5-6x+x^2)y \\
& \quad \quad \quad\quad\quad\quad \quad\quad\quad -24 p^2 (5 - 6 x + x^2) y\bigg).
\end{align*}
Now $\dfrac{\partial N}{\partial y}=0$ gives
\begin{equation*}
y=y_1:=\dfrac{5 p^3 + 6 p x (4-p^2)(1+2x)}{24 (4- p^2) (6 x - x^2-5)}.
\end{equation*}
Since $y_1$ must be a member of $(0,1)$ for critical points to exist, this is only possible if
\begin{equation}
24 (20 + (p-24) x + (4 + 2 p - p^2) x^2)+p^3 (5 - 6 x - 12 x^2)<24p^2(5-6x).\label{5c h1}
\end{equation}
Now, we find the solutions satisfying the inequality (\ref{5c h1})
for the existence of critical points using the hit and trial method. If we assume $p$ tends to 0 and 2, then no such $x\in (0,1)$ exists satisfying equation (\ref{5c h1}). Similarly, if we take $x$ tending to $0$ and $1$, then there does not exist any $p\in (0,2)$ satisfying equation (\ref{5c h1}). Therefore, we conclude that the function $N$ has no critical point in $(0,2)\times (0,1)\times (0,1).$

\item Now, we study the interior of each of the six faces of the cuboid $V$.\\
\noindent \underline{When $p=0$}, $N(p,x,y)$ becomes
\begin{equation}
c_1(x,y):=\dfrac{y^2(15 - 12 x^2 - 3 x^4) + 18 x (1 - y^2) - 2 x^3 (5-9 y^2)}{2160},\quad x,y\in (0,1).\label{5c 9.4}
\end{equation}
 Since
\begin{equation*}
\dfrac{\partial c_1}{\partial y}=\dfrac{y(1-x)^2(x+1)(5-x)}{360}\neq 0,\quad x,y\in (0,1),
\end{equation*}
we note that, in $(0,1)\times(0,1)$, $c_1$ does not have any critical point.	\\					
\noindent \underline{When $p=2$}, $N(p,x,y)$ settles into
\begin{equation}
N(2,x,y):=\dfrac{1}{20736},\quad x,y\in (0,1).\label{5c 9.3}
\end{equation}
						
\noindent \underline{When $x=0$}, $N(p,x,y)$ turns into
\begin{equation}
c_2(p,y):=\dfrac{(p^3 + 96 y - 24 p^2 y)^2}{1327104},\quad p\in (0,2)\quad \text{and}\quad y\in (0,1). \label{5c 9.1}
\end{equation}
We solve $\partial c_2/\partial p=0$ and $\partial c_2/\partial y=0$ to locate the points of maxima. On solving $\partial c_2/\partial y=0,$ we obtain
\begin{equation*}
	y=-\dfrac{p^3}{24(4-p^2)}(=:y_p).%\label{5c y}
\end{equation*}
Upon calculations, we observe that such $y_p$ does not belong to $(0,1)$. Consequently, no such critical point of $c_2$ exists in $(0,2)\times(0,1)$.\\
\noindent \underline{When $x=1$}, $N(p,x,y)$ becomes
\begin{equation}
N(p,1,y)=c_3(p,y):=\dfrac{24576 - 3264 p^2 - 2448 p^4 + 437 p^6}{6635520}, \quad p\in (0,2).\label{5c 9.2}
\end{equation}
And while computing $\partial c_3/\partial p=0$, we notice that $c_3$ has no critical point in $(0,2).$\\
\noindent \underline{When $y=0$}, $N(p,x,y)$ reduces to
\begin{align*}
c_4(p,x):&=\dfrac{1}{6635520}\bigg(6144 x (9 - 5 x^2) + 192 p^2 x (-144 + 15 x + 100 x^2 + 12 x^3)\\
&\quad \quad\quad \quad\quad \quad -48 p^4 x (-73 + 20 x + 80 x^2 + 24 x^3)\\
&\quad \quad\quad \quad \quad\quad+p^6 (5 - 12 x + 60 x^2 + 240 x^3 + 144 x^4)\bigg).
\end{align*}
Calculations lead to,
\begin{align*}
\dfrac{\partial c_4}{\partial x}&=\dfrac{1}{6635520}\bigg(-61440 x^2 - 6144 (-9 + 5 x^2) + 192 p^2 x (15 + 200 x + 36 x^2)\\
&\quad \quad\quad \quad \quad \quad -48 p^4 x (20 + 160 x + 72 x^2) +
 192 p^2 (-144 + 15 x + 100 x^2 \\
&\quad \quad\quad \quad \quad \quad+ 12 x^3)-48 p^4 (-73 + 20 x + 80 x^2 + 24 x^3) +
 p^6 (-12  \\
&\quad \quad\quad \quad \quad \quad + 120 x+ 720 x^2 + 576 x^3) \bigg)
\end{align*}
and \begin{align*}
\dfrac{\partial c_4}{\partial p}&=\dfrac{1}{6635520}\bigg(384 p x (-144 + 15 x + 100 x^2 + 12 x^3)-192 p^3 x (-73 + 20 x \\
&\quad\quad\quad\quad\quad \quad + 80 x^2+ 24 x^3)+6 p^5 (5 - 12 x + 60 x^2 + 240 x^3 + 144 x^4)\bigg).
\end{align*}

No solution exist for the system of equations, $\partial c_4/\partial x=0$ and $\partial c_4/\partial p=0$, according to a numerical calculation, in $(0,2)\times (0,1).$\\
						
\noindent \underline{When $y=1$}, $N(p,x,y)$ reduces to
\begin{align*}
c_5(p,x):&=\frac{1}{6635520}\bigg(5 p^6 +(4-p^2)( 12 p^4 x + 120 p^4 x^2 + 180 p^2 (4 - p^2) x^2 \\
& \quad\quad\quad\quad\quad\quad+1536 (4-p^2) x^3 + 240 p^2 (4-p^2) x^3 + 144 p^2 (4-p^2) x^4  \\
&\quad\quad\quad\quad\quad\quad+ 3456 (4-p^2)x (1 - x^2)+48 (1- x^2) (p^3 (5-6x \\
&\quad\quad\quad\quad\quad\quad- 12 x^2)+24 p x (1 + 2 x)))\bigg).
\end{align*}
The two equations $\partial c_5/\partial x=0$ and $\partial c_5/\partial p=0$ also do not assume any solution in $(0,2)\times (0,1).$
						
\item Next, we check the maximum values of $N(p,x,y)$ obtained on the edges of the cuboid $V$. From equation (\ref{5c 9.1}), we have $N(p,0,0)=t_1(p):=p^6/1327104.$ It is easy to observe that $t_1'(p)=0$ for $p=0$ %and $p=\delta_1:=1.43533$
in the interval $[0,2]$. The maximum value of $t_1(p)$ is $0.$
Now the equation (\ref{5c 9.1}) reduces to $N(p,0,1)=t_2(p):=(96 - 24 p^2 + p^3)^2/1327104$ at $y=1$. Since, $t_2'(p)<0$ in $[0,2]$, hence $p=0$ is the point of maxima. Thus
\begin{equation*}
N(p,0,1)\leq \dfrac{1}{144}, \quad p\in [0,2].
\end{equation*}
Through computations, equation (\ref{5c 9.1}) shows that $N(0,0,y)$ attains its maxima at $y=1.$ Hence
\begin{equation*}
N(0,0,y)\leq \dfrac{1}{144}, \quad y\in [0,1].
\end{equation*}
Since, the equation (\ref{5c 9.2}) is free from $x$, we have $N(p,1,1)=N(p,1,0)=t_3(p):=(24576 - 3264 p^2 - 2448 p^4 + 437 p^6)/6635520.$ Now, we observe that $t_3'(p)<0$ in $[0,2]$, consequently, $t_3(p)$ attains its maximum at $p=0$. Hence

\begin{equation*}
N(p,1,1)=N(p,1,0)\leq 0.0037037,\quad p\in [0,2].
\end{equation*}
On substituting $p=0$ in equation (\ref{5c 9.2}), we get, $N(0,1,y)=1/270.$ The equation (\ref{5c 9.3}) does not contain any variable such as $p$, $x$ and $y$. Therefore, the maxima of $N(p,x,y)$ on the edges %$p=2, x=1; p=2, x=0; p=2, y=0$ and $p=2, y=1,$ respectively,
is given by
\begin{equation*}
N(2,1,y)=N(2,0,y)=N(2,x,0)=N(2,x,1)=\dfrac{1}{20736},\quad x,y\in [0,1].
\end{equation*}

Using equation (\ref{5c 9.4}), we obtain $N(0,x,1)=t_4(x):=(15 - 12 x^2 + 8 x^3 - 3 x^4)/2160.$ Upon calculations, we see that $t_4$ is a decreasing function in $[0,1]$ and its maximum value is achieved at $x=0.$ Hence
\begin{equation*}
N(0,x,1)\leq \dfrac{1}{144},\quad x\in [0,1].
\end{equation*}
On again using equation (\ref{5c 9.4}), we get $N(0,x,0)=t_5(x):=x(9-5x^2)/1080.$ On further calculations, we get $t_5'(x)=0$ for $x=\beta_0:=\sqrt{3/5}.$ Also, $t_5(x)$ increases in $[0,\beta_0)$ and decreases in $(\beta_0,1].$ So, $\beta_0$ is the point of maxima. Thus
\begin{equation*}
N(0,x,0)\leq 0.00430331,\quad x\in [0,1].
\end{equation*}
\end{enumerate}
Because of all the cases discussed above, the inequality (\ref{5 sharph31c}) holds.\\
The function $f_2(z)\in \mathcal{C}_{e}$, defined as
\begin{equation*}
f_2(z)=\int_{0}^{z}\bigg(\exp\bigg(\int_{0}^{y}\dfrac{e^{t^3}-1}{t}dt\bigg)\bigg)dy=z+\dfrac{z^4}{12}+\dfrac{5z^7}{252}+\cdots,
					\end{equation*}
with $f_2(0)=f_2'(0)-1=0$, plays the role of an extremal function for the bounds of $|  H_{3}(1)| $ having values $a_3=a_5=0$ and $a_4=1/12.$
\end{proof}

\subsection{Fourth Hankel Determinant for $\mathcal{C}_{e}$}
In this part of the section, we derive the bounds of $H_4(1)$ including finding the bounds of sixth and seventh coefficients for functions in the class $\mathcal{C}_e$.				
By selecting $q=4$ and $n=1$ in the equation (\ref{5hqn}), the expression of $| H_{4}(1)| $ can be obtained for functions in the class $\mathcal{C}_{e}$, which is given as follows:
\begin{equation}\label{5 ch41}
H_{4}(1)=a_7H_{3}(1)-a_6U_1+a_5U_2-a_4U_3.
\end{equation}
Here
\begin{equation}\label{5cu1form}
U_1:=a_6(a_3-a_2^2)+a_3(a_2a_5-a_3a_4)-a_4(a_5-a_2a_4),
\end{equation}
\begin{equation}\label{5cu2form}
U_2:=a_3(a_3a_5-a_4^2)-a_5(a_5-a_2a_4)+a_6(a_4-a_2a_3),
\end{equation}
and
\begin{equation}\label{5cu3form}
U_3:=a_4(a_3a_5-a_4^2)-a_5(a_2a_5-a_3a_4)+a_6(a_4-a_2a_3).
\end{equation}

\noindent We start by determining the bounds for $U_1$, $U_2$, and $U_3$.\\
By substituting the values of $a_i$'s $(i=2,3,...,6)$ in (\ref{5cu1form}) from equations (\ref{5 ea2c})-(\ref{5 ea6c}), we obtain
\begin{align*}
132710400U_1&=487 p_1^7 - 6304 p_1^5 p_2 + 11440 p_1^3 p_2^2 - 24960 p_1 p_2^3 +5280 p_1^4 p_3  \\
& \quad + 34560 p_1 p_3^2+ 19200 p_1^2 p_2 p_3-53760 p_2^2 p_3 + 57600 p_1 p_2 p_4   \\
&\quad - 138240 p_3 p_4+ 184320 p_2 p_5-92160 p_1^2 p_5+ 8640 p_1^3 p_4,
\end{align*}
can also be viewed as the following, due to the triangle inequality,
\begin{align*}
132710400| U_1| &\leq | p_1^5(487 p_1^2 - 6304  p_2)|  + | p_1 p_2^2(11440 p_1^2 - 24960 p_2)| \nonumber \\
& \quad+ | p_1 p_3(5280 p_1^3+ 34560p_3)| +|  p_2 p_3 (19200 p_1^2-53760 p_2) |   \nonumber\\
&\quad  + |  p_4(57600 p_1 p_2 - 138240 p_3)| + | p_5(184320 p_2  -92160 p_1^2)| \nonumber\\
&\quad  + | 8640 p_1^3 p_4| .%\label{5 ecu1}
\end{align*}
Using Lemma \ref{2 pomi lemma}, we arrive at
%\begin{align}
%| p_1^5(487 p_1^2 - 6304 p_2)|  \leq 40320, \quad    | p_1 p_2^2(11440 p_1^2 - 24960 p_2)| \leq 399360,\label{5 ecu11}
%\end{align}
%\begin{align}
%| p_1 p_3(5280 p_1^3+ 34560p_3)|  \leq 445440, \quad  |  p_2 p_3( 19200 p_1^2-53760 p_2) |  \leq 430080, \label{5 ecu12}
%\end{align}
%\begin{align}
%|  p_4(57600 p_1 p_2 -138240 p_3)| \leq 55296, \quad | p_5(184320 p_2-92160 p_1^2)| \leq 73728, \label{5 ecu13}
%%\end{align}
%and\begin{equation}  | 8640 p_1^3 p_4|  \leq 138240.\label{5 ecu14}\end{equation}
%By substituting the values obtained from equations (\ref{5 ecu11})-(\ref{5 ecu14}) in (\ref{5 ecu1}),
\begin{align*}
| U_1| &\leq %\frac{1582464}{132710400}
\frac{4121}{345600}\approx 0.0119242.%\label{5 ecu1value}
\end{align*}

% BOUND OF U2
We replace the values of $a_i$'s $(i=2,3,...,6)$ from equations (\ref{5 ea2c})-(\ref{5 ea7c}) in equation (\ref{5cu2form}) and proceed on the same lines to obtain the bound of $U_2$
\begin{align*}
1592524800U_2&= 463 p_1^8 - 2732 p_1^6 p_2 - 23472 p_1^4 p_2^2 - 14400 p_1^2 p_2^3 + 14592 p_1^5 p_3  \\
&\quad - 108288 p_1^2 p_3^2+92928 p_1^3 p_2 p_3- 138240 p_1 p_2^2 p_3 +1105920 p_3 p_5 \\
&\quad - 25344 p_1^4 p_4 +276480 p_2^2 p_4 - 995328 p_4^2+ 373248 p_1^2 p_2 p_4\\
&\quad - 276480 p_1 p_2 p_5+ 221184 p_1 p_3 p_4 - 161280 p_1^3 p_5- 322560 p_2 p_3^2,
\end{align*}
by implementing the triangle inequality,
\begin{align*}
1592524800| U_2| &\leq | p_1^6( 463 p_1^2 - 2732p_2)|  + |  p_1^2 p_2^2(- 23472 p_1^2 - 14400 p_2)| \nonumber\\
&\quad +|  p_1^2 p_3(14592 p_1^3 - 108288 p_3)| +|  161280 p_1^3 p_5| \nonumber \\
&\quad+ | p_1^2p_4(373248 p_2- 25344 p_1^2) | +|  p_4(276480 p_2^2 - 995328 p_4)| \nonumber\\
&\quad+| 322560 p_2 p_3^2|  + | p_1p_2 p_3(92928 p_1^2 - 138240p_2)|  \nonumber\\
&\quad+| 221184 p_1 p_3 p_4|+| p_5(1105920 p_3- 276480 p_1 p_2) | . %\label{5 ecu2}
\end{align*}
				
By applying Lemma \ref{2 pomi lemma}, we have
\begin{align*}
| U_2| &\leq \frac{24947200 + 866304 \sqrt{\frac{282}{61}}}{1592524800}\approx 0.0168348.%\label{5 ecu2value}
\end{align*}
%\begin{equation}
%\begin{aligned}
 %\left.
%\begin{array}{cc}
%    & |  p_1^6(463 p_1^2 - 2732  p_2)| \leq 349696;\\
% &| p_1^2 p_2^2( 23472 p_1^2+14400 p_2)| \leq 1963008;\\
%& | p_1p_2 p_3(92928 p_1^2- 138240p_2)| \leq 2211840;\\
%& |  p_4(276480 p_2^2 - 995328 p_4)| \leq 3981312;\label{5 ecu21}
%\end{array}
%\right\}
%\end{aligned}
%\end{equation}
%and
%\begin{align}
%|  p_1^6(463 p_1^2 - 2732  p_2)| \leq 349696,\quad | p_1^2 p_2^2( 23472 p_1^2+14400 p_2)| \leq 1963008,\label{5 ecu21}
%\end{align}
%\begin{equation}
%\begin{aligned}
 %\left.
%\begin{array}{cc}
   % & |  p_1^6(463 p_1^2 - 2732  p_2)| \leq 349696;\\
% &|  p_5(1105920 p_3- 276480 p_1 p_2)| \leq 4423680;\\
%& | p_1^2p_4(373248 p_2- 25344 p_1^2)| \leq 5971968;\\
%& | p_1^2 p_3(14592 p_1^3 - 108288 p_1^2 p_3^2| \leq 866304\sqrt{\frac{282}{61}};\\
%&| 221184 p_1 p_3 p_4| +| 161280 p_1^3 p_5| +| 322560 p_2 p_3^2| \leq  6045696.\label{5 ecu22}
%\end{array}
%\right\}
%\end{aligned}
%\end{equation}

%\begin{align}
%| p_1^2 p_3(14592 p_1^3 - 108288 p_1^2 p_3^2| \leq 866304\sqrt{\frac{282}{61}},\quad  \label{5 ecu22}
%\end{align}

%\begin{align}
%| p_1^2p_4(373248 p_2- 25344 p_1^2)| \leq 5971968, \quad \label{5 ecu23}
%\end{align}
%\begin{align}
%|  p_5(1105920 p_3- 276480 p_1 p_2)| \leq 4423680,   \label{5 ecu24}
%\end{align}
%and
%\begin{align}
%| 221184 p_1 p_3 p_4| +| 161280 p_1^3 p_5| +| 322560 p_2 p_3^2| \leq  6045696.\label{5 ecu25}
%\end{align}
%By substituting the values from equations (\ref{5 ecu21}) and (\ref{5 ecu22}) in (\ref{5 ecu2}),			
Again, substitute the values of $a_i$'s $(i=2,3,...,6)$ from equations (\ref{5 ea2c})-(\ref{5 ea7c}) in (\ref{5cu3form}) and proceed to calculate the bound of $U_3$ in the same manner.			
\begin{align*}
38220595200U_3&=11424 p_1^8- 128256 p_1^6 p_2 + 10812 p_1^7 p_2 - 503 p_1^9  + 69120 p_1^4 p_2^2  \\
&\quad + 552960 p_1^2 p_2^3- 42192 p_1^5 p_2^2- 181440 p_1^3 p_2^3+206208 p_1^4 p_2 p_3\\
&\quad -11664 p_1^6 p_3+ 1889280 p_1^3 p_2 p_3-1658880 p_1 p_2^2 p_3  - 2211840 p_1^2 p_3^2\\
&\quad -2211840 p_2 p_3^2 + 283392 p_1^3 p_3^2- 967680 p_1 p_2 p_3^2 + 3317760 p_1 p_3 p_4    \\
&\quad- 483840 p_1^4 p_4+ 1271808 p_1^3 p_2 p_4 - 117504 p_1^5 p_4 + 1658880 p_1 p_2^2 p_4\\
&\quad - 5971968 p_1 p_4^2 + 6635520 p_2 p_3 p_4 - 331776 p_1^2 p_3 p_4+ 26542080 p_3 p_5 \\
&\quad-6635520 p_1 p_2 p_5+ 244224 p_1^5 p_3- 794880 p_1^2 p_2^2 p_3- 2764800 p_3^3\\
&\quad- 829440 p_1^2 p_2 p_4- 3870720 p_1^3 p_5,
\end{align*}
can be visualized as the following with the help of the triangle inequality,
\begin{align*}
38220595200| U_3| &\leq |  p_1^6(11424 p_1^2- 128256 p_2)| +|  p_1^7(10812 p_2 - 503 p_1^2) | \nonumber\\
&\quad+ |  p_1^2 p_2^2(69120 p_1^2  + 552960p_2)|  + |  p_1^3 p_2^2( 42192 p_1^2+ 181440 p_2)| \nonumber\\
&\quad+| p_1^4 p_3(206208p_2-11664 p_1^2 )| + | p_1 p_2p_3(1889280 p_1^2-1658880p_2)|  \nonumber\\
&\quad+| p_3^2(2211840 p_1^2 +2211840 p_2)|  +| p_1 p_3^2(283392 p_1^2- 967680 p_2)| \nonumber \\
&\quad +| p_1 p_4(3317760p_3 - 483840 p_1^3) | +| p_1^3 p_4(1271808p_2-117504 p_1^2 )| \nonumber\\
&\quad+| p_1p_4(1658880p_2^2- 5971968 p_4)| +| p_3 p_4 (6635520 p_2 - 331776 p_1^2) | \nonumber\\
&\quad+ | p_5(26542080 p_3-6635520 p_1 p_2)| +|  244224 p_1^5 p_3- 794880 p_1^2 p_2^2 p_3\nonumber\\
&\quad- 2764800 p_3^3 - 829440 p_1^2p_2 p_4- 3870720 p_1^3 p_5|.%.\label{5 ecu3}
\end{align*}
By applying Lemma \ref{2 pomi lemma}, we get
\begin{align*}
| U_3| &\leq \frac{560108544 + 106168320 \sqrt{\frac{3}{41}}}{38220595200}\approx0.015406.%\label{5 ecu3value}
\end{align*}

\begin{remark}\label{5cu}
The bounds of $U_1$, $U_2$ and $U_3$, based on the above calculations, are	$0.0119242$, $0.0168348$, and $0.015406$ respectively.
\end{remark}
			
The bounds of $a_i$ $(i=2,3,4,5)$ for functions in the class $\mathcal{C}_{e}$ are obtained in \cite{zap2019}, presented below in the following remark:
\begin{remark}\label{zapavaluesc}
For $f\in \mathcal{C}_{e},$ $| a_2| \leq 1/2,$ $| a_3| \leq 1/4$, $| a_4| \leq 17/144$ and $| a_5| \leq 5/72$. The first three bounds are sharp.
\end{remark}

Next, we calculate the bounds of the sixth and seventh coefficient of functions belonging to the class $\mathcal{C}_{e}$ to establish our main result along the lines of Lemma \ref{5 a6a7lemma}.

\begin{lemma}\label{5 ca6a7lemma}
Let $f\in \mathcal{C}_{e}.$ Then $| a_6| \leq 587/10800\approx 0.0543519$ and $| a_7| \leq 0.0343723$.
\end{lemma}
\begin{proof}
A suitable rearrangement of the terms given in equation (\ref{5 ea6c}) provides us
\begin{equation*}
345600a_6= 5760 p_5- 480 p_2 p_3 +720 p_1 p_4- 480 p_1 p_2^2-17 p_1^5 + 220 p_1^3 p_2- 480 p_1^2 p_3. \\
%&\quad - 480 p_1^2 p_3 -17 p_1^5 .
\end{equation*}
Further, through the triangle inequality, it can be viewed as
\begin{align*}
345600| a_6| &\leq  | 5760 p_5- 480 p_2 p_3| +| p_1(720p_4- 480 p_2^2) | +| 17 p_1^5|  \nonumber\\
&\quad +| p_1^2( 220 p_1 p_2- 480  p_3)| .%\label{5 ca6}
\end{align*}
Using Lemma \ref{2 pomi lemma}, we arrive at
\begin{equation*}
| a_6| \leq \frac{587}{10800}\approx 0.0543519.
%| a_6| \leq \frac{18784}{57600}\approx 0.326111.
\end{equation*}
%\begin{equation}
%%\begin{equation}
%| p_1^2( 220 p_1 p_2- 480  p_3)| \leq 3840,\quad \text{and}\quad |  17 p_1^5| \leq 544. \label{5 ca62}
%\end{equation}
%By using equation (\ref{5 ca61}) and (\ref{5 ca62}) in (\ref{5 ca6}), we have

Similarly, considering equation (\ref{5 ea7c}), we have
\begin{align*}
58060800a_7&=881 p_1^6 - 13260 p_1^4 p_2 + 48240 p_1^2 p_2^2 -106560 p_1 p_2 p_3 + 29040 p_1^3 p_3 \nonumber\\
&\quad- 57600 p_3^2 + 69120 p_1 p_5- 56160 p_1^2 p_4 - 86400 p_2 p_4 - 14400 p_2^3.
\end{align*}
It can also be seen as with the aid of the triangle inequality,
\begin{align}
58060800| a_7| &\leq |  p_1^4 (881 p_1^2 - 13260p_2)|  + |  p_1 p_2(48240 p_1 p_2-106560 p_3)|   \nonumber\\
&\quad + | p_3(29040 p_1^3 - 57600 p_3)| +|  p_1(69120 p_5- 56160 p_1 p_4)| \nonumber\\
&\quad+|  p_2(86400 p_4+14400 p_2^2)| .\label{5ca7}
\end{align}
Lemma \ref{2 pomi lemma} takes us at
%\begin{align}
%|  p_1^4(881 p_1^2 - 13260 p_2)| \leq 424320,\quad  | p_1 p_2(48240 p_1 p_2-106560  p_3)|  \leq 852480, \label{5 eca71}
%\end{align}
%\begin{align}\label{5 eca72}
%| p_3(29040 p_1^3 - 57600 p_3)| \leq 921600\sqrt{\frac{15}{119}}, \quad | p_1(69120p_5- 56160 p_1 p_4 )| \leq 276480,
%\end{align}
%and \begin{align}\label{5 eca73}
%|  p_2(86400  p_4 + 14400 p_2^2)| \leq 460800.
%\end{align}
%By substituting the values from equations (\ref{5 eca71})-(\ref{5 eca73}) in (\ref{5ca7}), we have
\begin{equation*}
| a_7| \leq \frac{2014080 + 921600 \sqrt{\frac{15}{119}}}{58060800}\approx 0.0403246.
\end{equation*}
\end{proof}

We obtain the following result by omitting the proof as it directly follows from Theorem \ref{5 ceh31}, %equations (\ref{5 ecu1value})-(\ref{5 ecu3value})
Remark \ref{5cu}, Remark \ref{zapavaluesc}, Lemma \ref{5 ca6a7lemma} and equation (\ref{5 ch41}).
\begin{theorem}\label{5 ch41bound}
Let $f\in \mathcal{C}_{e}.$ Then
\begin{equation*}
| H_{4}(1)| \leq 0.00101775.\end{equation*}
\end{theorem}

%On Acknowledgment %%%%%%%%
%		\section*{Acknowledgement}
	%		The author would like to thank the referees for the helpful
				%		suggestions.
				%		
				%%%% Bibliography  %%%%%%%%%%
\subsection*{Acknowledgment}
Neha is thankful to the Department of Applied Mathematics, Delhi Technological University, New Delhi-110042 for providing Research Fellowship.

\end{document}